\documentclass[a4paper,11pt,twoside]{amsart}

\usepackage[english]{babel}
\usepackage[latin1]{inputenc}
\usepackage[T1]{fontenc}
\usepackage{amsfonts}
\usepackage{amsmath}
\usepackage{amsthm}
\usepackage{amscd}
\usepackage{float}
\usepackage{enumerate}
\usepackage{multirow}
\usepackage{graphicx}
\usepackage{latexsym}
\usepackage{amssymb}
\usepackage[all]{xy}
\usepackage{mathrsfs}
\usepackage{url}

\theoremstyle{plain}
\newtheorem{teor}{Theorem}[section]
\newtheorem{cor}[teor]{Corollary}

\newtheorem{prop}[teor]{Proposition}
\newtheorem{lemma}[teor]{Lemma}
\newtheorem{defn}[teor]{Definition}

\theoremstyle{definition}
\newtheorem{oss}[teor]{Remark}

\newfloat{tabella}{H!}{tabella}
\floatname{tabella}{Tabella}
\newfloat{immagine}{H}{imm}
\floatname{immagine}{Figura}

\newcommand{\A}{\mathcal{A}}

\newcommand{\Csharp}{\mathscr{C}}
\newcommand{\C}{\mathbb{C}}
\newcommand{\D}{\mathscr{D}}

\newcommand{\E}{\mathscr{E}}
\newcommand{\F}{\mathscr{F}}

\newcommand{\Psharp}{\mathscr{P}}

\newcommand{\R}{\mathbb{R}}

\newcommand{\T}{\mathcal{T}}
\newcommand{\Z}{\mathbb{Z}}

\newcommand{\ob}{\operatorname{Ob}}

\renewcommand{\hom}{\operatorname{Hom}}
\newcommand{\spec}{\operatorname{spec}}
\newcommand{\id}{\operatorname{id}}

\newcommand{\st}{\text{ s.t. }}

\newcommand{\coh}{{\operatorname{Coh}}}
\newcommand{\qcoh}{{\operatorname{Qcoh}}}

\newcommand{\der}{{\mathrm{D}}}
\newcommand{\cone}{{\operatorname{Cone}}}

\newcommand{\dera}{{\der^b(A)}}
\newcommand{\perfa}{{\mathbf{Perf}(A)}}
\newcommand{\perf}{{\mathbf{Perf}}}
\newcommand{\add}[1]{\operatorname{add}\left\{#1\right\}}
\newcommand{\eend}{{\operatorname{End}}}
\newcommand{\stab}{{\operatorname{Stab}}}
\newcommand{\costab}{{\operatorname{co-Stab}}}
\newcommand{\gl}{{\operatorname{GL}}}

\begin{document}
\title{Fourier-Mukai functors and perfect complexes on dual numbers}
\author{FRANCESCO AMODEO AND RICCARDO MOSCHETTI}
\address{F.A.: Dipartimento di Matematica ``F.
	Enriques'', Universit{\`a} degli Studi di Milano, Via Cesare Saldini
	50, 20133 Milano, Italy}
	\email{francesco.amodeo@unimi.it}

\address{R.M.: Dipartimento di Matematica ``F. Casorati'', Universit{\`a}
	degli Studi di Pavia, Via Ferrata 1, 27100 Pavia, Italy}
	\email{riccardo.moschetti@unipv.it}


\maketitle

\begin{abstract}
We show that every exact fully faithful functor from the category of perfect complexes on the spectrum of dual numbers to the bounded derived category of a noetherian separated scheme is of Fourier-Mukai type. The kernel turns out to be an object of the bounded derived category of coherent complexes on the product of the two schemes. We also study the space of stability conditions on the derived category of the spectrum of dual numbers.
\end{abstract}
\section{Introduction}

Fourier-Mukai functors play an important role in many geometric contexts. For example, if $S$ is a projective K3 surface, then any other K3 surface $Y$ for which there exists a Fourier-Mukai equivalence with kernel $\mathcal{E}$, $\Phi_{\mathcal{E}}:\der^b(S)\longrightarrow \der^b(Y)$ is isomorphic to a moduli space of stable sheaves on $S$ (\cite{O}). One of the main results about Fourier-Mukai functors states that, if $X$ and $Y$ are smooth projective varieties and $F:\der^b(X)\longrightarrow \der^b(Y)$ is an exact fully faithful functor, then $F$ is of Fourier-Mukai type. Also, the kernel is unique up to isomorphisms (\cite{O}).\\ It was expected the same for \emph{every} exact functor, until the counterexample found by Rizzardo and Van den Bergh in \cite{RV}, in which they find a non fully faithful functor between the derived categories of coherent sheaves. More generally, one would analyze the following question:\\
\emph{Given a functor between the derived categories of two varieties, what are the minimal hypotheses to guarantee such a functor being of Fourier-Mukai type?} 

One could try to weak the hypothesis either on the functor, for example as in \cite{R}, or on the varieties. The latter is the case that interests us, and in particular we drop the smoothness hypothesis. Let $X$ be a projective scheme, and let $\perf(X)$ be the subcategory of $\der(\qcoh(X))$ consisting of the objects which are quasi-isomorphic to bounded complexes of locally free sheaves of finite type on $X$. The category $\perf(X)$ is always included in $\der^b(X)$, in view of the natural equivalence $\der^b(X) \cong \der^b_{\operatorname{coh}}(\qcoh(X))$ where $\der^b_{\operatorname{coh}}(\qcoh(X))$ is the full subcategory of $\der^b(\qcoh(X))$ consisting of objects with coherent cohomology. If the scheme $X$ is smooth, then the subcategory $\perf(X)$ coincides with $\der^b(X)$.
Let now $Y$ be a noetherian separated scheme. We say a functor $F:\perf(X)\longrightarrow \der^b(Y)$ to be of Fourier-Mukai type if there exists an object $\mathcal{E}\in \der^b(X\times Y)$ called kernel of the functor, such that $F\cong \Phi_{\mathcal{E}}$, with: $$\Phi_{\mathcal{E}}:\perf(X)\longrightarrow \der^b(Y),\quad\quad\quad \Phi_{\mathcal{E}}(-):=\mathbf{R}(p)_*(\mathcal{E}\stackrel{\mathbf{L}}{\otimes}  q^*(-))$$
where $p:X\times Y\longrightarrow Y$ and $q:X\times Y\longrightarrow X$ are the projections.\\\\
A highly non-trivial result by Lunts and Orlov in \cite{LO} shows, by the use of DG categories, that if $X$ is projective such that the maximal torsion subsheaf of dimension zero $T_0(\mathcal{O}_X)\subset \mathcal{O}_X$ is trivial, $Y$ is noetherian and $F:\perf(X)\longrightarrow \der^b(Y)$ is an exact fully faithful functor, then $F$ is of Fourier-Mukai type.\\\\
The hypothesis $T_0(\mathcal{O}_X)=0$ is related with the use of ample sequences and it seems not to be a very natural assumption. What happens if we consider a projective scheme $X$ such that $T_0(\mathcal{O}_X)\neq0$?
The simplest example of such scheme is given by $\spec k$. Here the result is trivial (see \cite{CS} Remark 2.2) in view of the simple description of $\der^b(\spec k)$.
Thus, we could take in consideration a zero dimensional non-smooth scheme. In such way, the maximal torsion subsheaf of dimension zero is certainly not trivial. A basic model of such type of objects is given by the "double point scheme", which is the spectrum of the ring of dual numbers $A:=k[\epsilon]/(\epsilon^2)$. Our main result is the following:

\begin{teor}
Let $Y$ be a quasi-compact and separated scheme. Let: $$F:\xymatrix{\perfa \ar[r] & \der(\qcoh Y)}$$ be a fully faithful functor. Then
there is an object $\mathcal{E}\in
\der(\qcoh (\spec A \times Y))$ such that:
$${\Phi_{\mathcal{E}}}|_{\perfa}\cong F.$$
Furthermore, if $Y$ is noetherian and $F$ sends $\perfa$ to $\der^b(Y)$, then
$$\mathcal{E}\in \der^b(\spec A \times Y).$$
\end{teor}
Thus we show that the main result in \cite{LO} still holds in a case in which the maximal torsion subsheaf of dimension zero is not trivial, hence we do expect it is possible to avoid this hypothesis and prove the same result in a more general case.

In the last two sections we deal with the problem of classifying all the stability conditions on the category $\dera$. The main result is the following:

\begin{teor}
$\stab(\dera)$ is isomorphic to $\C$, the universal covering of $\C^*$.
\end{teor}

In order to prove the results concerning such a classification, we will exploit the study, developed in Sections $3$-$4$, on the category $\dera$ following an argument originally used by J{\o}rgensen and Pauksztello in \cite{COSTAB}, Holm, J{\o}rgensen and Yang in \cite{SPARSENESS} for the category $\perfa$.
\\\\
In Section $2$ we recall some results about the indecomposable elements of $\perfa$. Actually, it turns out that they are quite simple and manageable, and this allows us to make some concrete calculations.\\
In Section $3$ we classify the morphisms between indecomposable elements; as we are in a $k$-linear category, we describe the generators of the spaces of morphisms. Afterwards, in Section $4$, we show how the compositions between those morphisms works. \\
In Section $5$ we focus our attention on fully faithful endofunctors of $\perfa$. We prove that every exact fully faithful functor $F:\perfa\longrightarrow \perfa$ is an equivalence. More precisely, it is isomorphic to the composition of a shift and a push forward along an automorphism of $\spec A$. Also, we give a concrete example of an autoequivalence of $\perfa$ that is not exact. \\
In Section $6$ we recall some definitions about DG categories which will be used in the proof of the main theorem.
Eventually, Sections $7$ and $8$ are devoted to the study the $t$-structures on the category $\dera$ and the space $\stab(\dera)$ of the stability conditions.

\section{Indecomposable complexes of $\dera$}
Let $k$ be a field and consider the ring $A=k[\epsilon] / (\epsilon^2)$=$\{a+\epsilon b \st a,b\in k\}$. The spectrum of $A$ consists of a single point, which corresponds to the maximal ideal $(\epsilon)$.
We are interested in studying the subcategory $\perf(\spec A) \subset \der^b(\spec A)$. The category $\perf(\spec A)$ coincides with the full subcategory of compact objects: an object $X\in \der^b(\spec A)$ is compact if, for every collection ${Y_i}_{\{i\in I\}}$ of objects in $\der^b(\spec A)$, the natural morphism:
$$\bigoplus \mathrm{Hom}(X,Y_i) \longrightarrow \mathrm{Hom}(X,\oplus Y_i) $$
is an isomorphism. \\
Let $\perfa$ be the full subcategory of $\dera := \der^b(A-\operatorname{mod}_{\operatorname{fg}})$ consisting of bounded complexes of finitely generated projective modules. Since $\coh(\spec A)$ is equivalent to $A-\operatorname{mod}_{\operatorname{fg}}$, then $\perf(\spec A)$ and $\der^b(\spec A)$ are equivalent to $\perfa$ and $\dera$.
A way to study the objects and mophisms of $\dera$ is by focusing on the indecomposable complexes. In an additive category, $X$ is indecomposable if $X \cong Y\oplus Z$ implies $Y \cong 0$ or $Z \cong 0$.
A good context to study indecomposable objects is provided by Krull-Schmidt categories, which are explained in details in \cite{M}.

\begin{defn}
Let $\Csharp$ be an additive category such that $\eend_\Csharp(X)$ is a semiperfect ring for all $X \in \Csharp$ (in that case $\Csharp$ is called a pre-Krull-Schmidt category). $\Csharp$ is called a Krull-Schmidt category if every idempotent splits, i.e. for every $X$ in $\Csharp$ and for every $e \in \eend_\Csharp(X)$ such that $e^2=e$, there exist $Y$ in $\Csharp$ and two morphisms $p:X \rightarrow Y$ and $q:Y \rightarrow X$ such that $qp=e$ and $pq=1_Y$.
\end{defn}
An additive category in which every idempotent splits is also called \emph{Karoubian}, hence a Krull-Schmidt category is a pre-Krull-Schmidt category that is also Karubian. Note that every abelian category is Karoubian.
In \cite{V} one can find another definition of the split property: an idempotent $e:X \rightarrow X$ splits if and only if there exists a non trivial decomposition $X \cong Y \oplus Z$ with $e$ corresponding to the projection on $Y$. These two definitions are equivalent in a triangulated category, which is the case of the present paper.
Thanks to the following result, proven in \cite{M}, we can confine ourselves to indecomposable elements:

\begin{teor}
In a Krull-Schmidt category every object can be decomposed into a finite direct sum of indecomposable objects. Moreover this decomposition is unique up to isomorphism.
\end{teor}

This theorem can be applied in our case thanks to the following

\begin{prop}
Let $X$ be a projective variety. Then $\perf(X)$ and $\der^b(X)$ are Krull-Schmidt categories.
\end{prop}
\begin{proof}
Since $X$ is projective, the endomorphism ring of every object of $\perf(X)$ and of $\der^b(X)$ is a finitely generated $k-$algebra of finite dimension, and then it is semiperfect. Moreover, $\der(\qcoh(X))$ is Karoubian, because it is a triangulated category with countably many direct sums. The subcategories $\perf(X)$ and of $\der^b(X)$ are thick and, then, Karoubian.
\end{proof}

\begin{prop} \label{prop:IndecomposableAndMaps}
Let $\Csharp$ be a Karoubian triangulated category, $\D$ an additive category and $F:\xymatrix{\Csharp \ar[r] & \D}$ a fully faithful additive functor. Then $F$ sends indecomposable objects of $\Csharp$ to indecomposable objects of $\D$.
\end{prop}
\begin{proof}
Let $X^{\bullet}$ be an indecomposable object of $\Csharp$. Since $\Csharp$ is Karoubian, $\hom(X^{\bullet},X^{\bullet})$ does not contain any idempotent except the identity and zero. Suppose $F(X^{\bullet}) \cong Y^{\bullet}\oplus Z^{\bullet}$, with $Y^{\bullet}$ and $Z^{\bullet}$ non zero. Since $F$ is fully faithful and additive we have an isomorphism of rings:
$$\hom(X^{\bullet},X^{\bullet}) \cong \hom(F(X^{\bullet}),F(X^{\bullet})) \cong \hom(Y^{\bullet}\oplus Z^{\bullet},Y^{\bullet}\oplus Z^{\bullet}).$$
The last space contains the projection $\xymatrix{Y^{\bullet}\oplus Z^{\bullet} \ar[r] & Y^{\bullet}}$, which is an idempotent different from the identity and zero, giving a contradiction.
\end{proof}

\begin{defn} \label{defn:Indecomposable}
For every $i\in \mathbb{N}$, $i>0$ let:
 $$X_i:=\{\xymatrix{
0\ar[r] & A_{(-i)}\ar[r]^\epsilon & \cdots \ar[r]^\epsilon & A_{(-1)} \ar[r] & 0
  }\}.$$

 $$X_\infty:=\{\xymatrix{
\cdots \ar[r]^\epsilon & A\ar[r]^\epsilon & \cdots \ar[r]^\epsilon & A_{(-1)} \ar[r] & 0
  }\}.$$
	Where $A_{(l)}$ stands for the module $A$ in the position $l \in \Z$.
\end{defn}

As proven in \cite{KU}, Section 3 or \cite{KY}, example 3.7, $\{X_i[h], i>0, h \in \Z\}$ are the indecomposable objects of $\perfa$ and $\{X_\infty[h], h \in \Z\}$ are the indecomposable objects of $\dera \smallsetminus \perfa$.

\section{Maps between indecomposable complexes}
In this section we will study in details the morphisms in the category $\perfa$, that is equivalent to the homotopy bounded category of complexes of finitely generated free $A-$modules. These results were already known, see for example \cite{SPARSENESS} and \cite{KY}, here we make the computation to describe explicitely the generators of the spaces of morphisms.

Notice that for every complexes $X_i$, $X_j$ and for every integers $\alpha$, $\beta$:
$$\hom(X_i[\alpha],X_j[\beta]) \cong \hom(X_i,X_j[\beta-\alpha]).$$
Now we study $\hom(X_i,X_j[\alpha])$ for a certain integer $\alpha$.

We start with the morphisms in $\perfa$ by considering the space $V:=\hom(X_i,X_j[\alpha])$ with $i,j \in \mathbb{N}$. \\
If $i > j$, there are five cases, from $\alpha \leq -j$ to $\alpha \geq i$.

\begin{enumerate}
	\item  $\alpha \leq -j$.	
\begin{equation*}
\xymatrix@=15pt{
0\ar[r]\ar[d] & A\ar[r]\ar[d] & \cdots\ar[r]\ar[d] & A\ar[r]\ar[d] & 0\ar[r]\ar[d]	& \cdots\ar[r]\ar[d] & \cdots\ar[r]\ar[d] & \cdots\ar[r]\ar[d] & \cdots\ar[r]\ar[d] & 0\ar[d]			\\
0\ar[r]       & \cdots\ar[r]       & \cdots\ar[r] & \cdots\ar[r]       & \cdots\ar[r] & 0\ar[r]       & A\ar[r]       & \cdots\ar[r] & A\ar[r]       & 0						
				 }
\end{equation*}

 It is clear that in this case all the vertical arrows are zero and thus $V=0$.\\
				
  \item $-j < \alpha \leq 0$. By the commutativity of the squares we have:
				
				$$\xymatrix@=15pt{
						 0\ar[r]\ar[d]& A\ar[r]\ar[d] & \cdots\ar[r]\ar[d] & A\ar[r]\ar[d]^{\epsilon b_1} & \cdots\ar[r] & A\ar[r]\ar[d]^{\epsilon b_k} & 0\ar[r]\ar[d] & \cdots\ar[r]\ar[d] & 0\ar[d]\\
						0\ar[r] & \cdots\ar[r] & 0\ar[r] & A\ar[r] & \cdots\ar[r] & A\ar[r] & \cdots\ar[r] & A\ar[r] & 0
				 } $$ 	
				
				 with $k=j+\alpha$. Define: $$B:=\sum_{l=1}^k (-1)^{l+1}b_{k-l+1}.$$
				Up to homotopy we can reduce the diagram to be the following one:
    
							$$\xymatrix@=15pt{
						 0\ar[r]\ar[d]& A\ar[r]\ar[d] & \cdots\ar[r]\ar[d] & A\ar[r]\ar[d]^0 & \cdots\ar[r] & A\ar[r]\ar[d]^{\epsilon B}  & 0\ar[r]\ar[d] & \cdots\ar[r]\ar[d] & 0\ar[d]\\
						0\ar[r] & \cdots\ar[r] & 0\ar[r] & A\ar[r] & \cdots\ar[r] & A\ar[r] & \cdots\ar[r] & A\ar[r] & 0
				 } $$ 	

This shows that in this case the space of morphisms $V$ is isomorphic to $k$. \\
				 \item $0 < \alpha < i-j$. By the commutativity of the squares we have:

				 $$\xymatrix@=15pt{
					0\ar[r]\ar[d] & A\ar[r]\ar[d] & \cdots\ar[r]\ar[d] & A\ar[r]\ar[d]^{\epsilon b_1} & \cdots\ar[r] & A\ar[r]\ar[d]^{\epsilon b_k} & \cdots\ar[r]\ar[d] & A\ar[r]\ar[d] & 0\ar[d]\\
					0\ar[r] & \cdots\ar[r] & 0\ar[r] & A\ar[r] & \cdots\ar[r] & A\ar[r] & 0\ar[r] & \cdots\ar[r] & 0
			 }$$
				
				 Up to homotopy the morphism is zero, hence $V=0$. \\

				 \item $i-j \leq \alpha < i$ and $\alpha \neq 0$. This case is similar to $(2)$. By the commutativity of the squares we have:

				$$\xymatrix@=15pt{
						0\ar[r]\ar[d] & \cdots\ar[r]\ar[d] & 0\ar[r]\ar[d] & A\ar[r]\ar[d]^{a+\epsilon b_1} & \cdots\ar[r] & A\ar[r]\ar[d]^{a+\epsilon b_k} & \cdots\ar[r]\ar[d] & A\ar[r]\ar[d] & 0\ar[d]\\
												 0\ar[r]& A\ar[r] & \cdots\ar[r] & A\ar[r] & \cdots\ar[r] & A\ar[r] & 0\ar[r] & \cdots\ar[r] & 0
				 } $$ 	
				
				Up to homotopy we can reduce the diagram to be the following one:
				
						$$\xymatrix@=15pt{
						0\ar[r]\ar[d] & \cdots\ar[r]\ar[d] & 0\ar[r]\ar[d] & A\ar[r]\ar[d]^{a} & \cdots\ar[r] & A\ar[r]\ar[d]^{a} & \cdots\ar[r]\ar[d] & A\ar[r]\ar[d] & 0\ar[d]\\
												 0\ar[r]& A\ar[r] & \cdots\ar[r] & A\ar[r] & \cdots\ar[r] & A\ar[r] & 0\ar[r] & \cdots\ar[r] & 0
				 } $$ 	
				
				Thus $V$ is still isomorphic to $k$.

				 \item $i \leq \alpha$. This case is similar to $(1)$. Thus, $V$ is equal to zero.

\end{enumerate}

If $i=j$, the calculations are similar to the previous case. Note that $(3)$ can not hold in this case. However, if $\alpha=0$, by the commutativity of the squares we obtain:
				 $$\xymatrix@=15pt{
						 0\ar[r]\ar[d] & A\ar[r]\ar[d]^{a+\epsilon b_1} & \cdots\ar[r] & A\ar[r]\ar[d]^{a+\epsilon b_h} & 0\ar[d]\\
						 0\ar[r] & A\ar[r] & \cdots\ar[r] & A\ar[r] & 0
				 }$$
Define: $$C:=\sum_{l=1}^i (-1)^{l+1}b_{i-l+1}.$$
Up to homotopy we can reduce the diagram to be the following one:
 $$\xymatrix@=15pt{
						 0\ar[r]\ar[d] & A\ar[r]\ar[d]^{a} & \cdots\ar[r]& A\ar[r]\ar[d]^{a+\epsilon C} & 0\ar[d]\\
						 0\ar[r] & A\ar[r] & \cdots\ar[r] & A\ar[r] & 0
				 }$$ This shows that in this case the space of morphisms $V$ is equal to $k\oplus k$.\\

If $i<j$, the calculations are similar to the case $i>j$.

We can sum it all up in the following:

\begin{prop} \label{prop:ClassificationMorphisms}
Consider the space $V=\hom(X_i,X_j[\alpha])$:
\begin{itemize}
	\item If $-j < \alpha \leq \min\{0,i-j\}$ and $(i-j, \alpha) \neq (0,0)$ then $V$ has dimension $1$ and it is generated by $\epsilon^i_{j[\alpha]}$. These morphisms are called of $k_\epsilon$-type.
	\item If $\max\{0,i-j\} \leq \alpha < i$ and $(i-j, \alpha) \neq (0,0)$ then $V$ has dimension $1$ and it is generated by $1^i_{j[\alpha]}$. These morphisms are called of $k_1$-type.
	\item If $i=j$ and $\alpha=0$ then $V$ has dimension $2$ and it is generated by both $\epsilon^i_{i[0]}$ and $1^i_{i[0]}$. These morphisms are called of $k^2$-type
	\item $V=\{0\}$ for all the remaining cases.
\end{itemize}
A morphism between two indecomposable objects can be described by the couple $(a,b)$ of elements of $k$, where $a$ is the coefficient of the generator $1$ and $b$ is the coefficient of the generator $\epsilon$.
\end{prop}

Note that the morphisms of $k_1$-type and $k_\epsilon$-type correspond, respectively, to the morphisms in $F^+$ and in $F^-$ as described in \cite{SPARSENESS}.

\begin{cor} \label{cor:gradedring}
The graded $k$-vector space $\hom^*(X_i,X_i)$ is uniquely determined by $i$
\end{cor}
\begin{proof}
If we have $\hom^*(X_i,X_i) \cong \hom^*(X_j,X_j)$, is straightforward to prove that $i=j$ by the result of the previous proposition.
\end{proof}

\begin{oss} \label{oss:NonPerfectIndecomposable}
The results of this proposition can be extended to $\dera$; one can easly prove that
\begin{itemize}
\item $\hom(X_\infty,X_\infty[h])$ is generated by $1$ if $h \geq 0$ and is $0$ otherwise.
\item $\hom(X_\infty,X_i[h])$ is generated by $\epsilon$ if $-i < h \leq 0$ and is $0$ otherwise.
\item $\hom(X_i,X_\infty[h])$ is generated by $1$ if $0 \leq h < i$ and is $0$ otherwise.
\end{itemize}
\end{oss}

As a consequence of the previous proposition, for all $X^\bullet$ and $Y^\bullet$ in $\perfa$ there is the following isomorphism:
$$\mathrm{Hom}_{\perfa}(X^{\bullet}, Y^{\bullet}) \cong \mathrm{Hom}_{\perfa}(Y^{\bullet}, X^{\bullet}).$$
More generally, Serre duality holds in $\perfa$. This is a particular case of Theorem 6.7 in \cite{H}.

\section{Compositions}
We are now wondering how the composition of morphisms in $\perfa$ works; that is, given two morphisms between indecomposable objects:
$$f:\xymatrix{
X_i \ar[r] & X_j[\alpha]
} \text{ and } g:\xymatrix{
X_j[\alpha] \ar[r] & X_k[\beta],
}$$
we are asking for what type the morphism $g\circ f$ is. We are going to study the compositions of the generators of the morphism described in Proposition \ref{prop:ClassificationMorphisms}. The situation is summed up in the following table.
Clearly, if either $f$ or $g$ is the zero morphism, then also the composition $g \circ f$ is zero.

$$\begin{tabular}{c||c|c|c}

  $\circ$    &    $0$     &     $1^i_{j[\alpha]}$    &    $\epsilon^i_{j[\alpha]}$ \\
\hline \hline
  $0$            &    $0$     &     $0$                              &     $0$   \\
\hline
$1^{j[\alpha]}_{k[\beta]}$ & $0$ & (i) & (ii) \\
\hline
$\epsilon^{j[\alpha]}_{k[\beta]}$ & $0$ & (iii) & (iv) \\
\end{tabular}$$\\

Proposition \ref{prop:ClassificationMorphisms} gives the conditions for the generators to be well defined. \\
Case (i) holds when $\max\{0,i-j\} \leq \alpha < i$ and $\max\{0,j-k\} \leq \beta-\alpha < j$. \\
Case (ii) holds when $-j < \alpha \leq \min\{0,i-j\}$ and $\max\{0,j-k\} \leq \beta-\alpha < j$. \\
Case (iii) holds when $\max\{0,i-j\} \leq \alpha < i$ and $-k < \beta-\alpha \leq \min\{0,j-k\}$. \\
Case (iv) holds when $-j < \alpha \leq \min\{0,i-j\}$ and $-k < \beta-\alpha \leq \min\{0,j-k\}$.

\begin{enumerate}
\item[(i)] The composition of $1^{j[\alpha]}_{k[\beta]} \circ 1^i_{j[\alpha]}$ is a morphism from $X_i$ to $X_k[\beta]$. If $\max\{0,i-k\} \leq \beta < i$ holds, that is the condition of having a morphism of $k_1$-type between $X_i$ and $X_k[\beta]$, then $1^{j[\alpha]}_{k[\beta]} \circ 1^i_{j[\alpha]}= 1^i_{k[\beta]}$, as shown in the following diagram:

				 $$\xymatrix@=15pt{
						 &        &        & A\ar[r]\ar[d]^{1} & \cdots\ar[r]\ar[d]^{1} & \cdots\ar[r]\ar[d]^{1} & \cdots\ar[r]\ar[d]^{1} & A \\
						 & A\ar[r]\ar[d]^1      & \cdots\ar[r]\ar[d]^1 & A\ar[r]\ar[d]^{1} & \cdots\ar[d]^1\ar[r] & \cdots\ar[d]^1\ar[r] & A      &    \\
					 A\ar[r] & \cdots\ar[r] & \cdots\ar[r] & A\ar[r] & \cdots\ar[r] & A.      &        &	
				 }$$

Otherwise $1^{j[\alpha]}_{k[\beta]} \circ 1^i_{j[\alpha]}=0$\\

\item[(ii)] The composition $1^{j[\alpha]}_{k[\beta]} \circ \epsilon^i_{j[\alpha]}$ is a morphism from $X_i$ to $X_k[\beta]$. If $-k < \beta \leq \min\{0,i-k\}$ holds, that is the condition of having a morphism of $k_\epsilon$-type between $X_i$ and $X_k[\beta]$, then $1^{j[\alpha]}_{k[\beta]} \circ \epsilon^i_{j[\alpha]}=\epsilon^i_{k[\beta]}$, as shown in the following diagram:

 $$\xymatrix@=15pt{
A\ar[r] & \cdots\ar[r] & \cdots\ar[r]\ar[d]^0 & \cdots\ar[r]\ar[d]^0 & A\ar[d]^{\epsilon} & & & \\
& & A\ar[r]\ar[d]^1 & \cdots\ar[r]\ar[d]^1 & A\ar[r]\ar[d]^{1} & \cdots\ar[r]\ar[d]^1 & \cdots\ar[r]\ar[d]^1 & A \\
& A\ar[r] & \cdots\ar[r] & \cdots\ar[r] & A\ar[r] & \cdots\ar[r] & A.	
				 }$$\\

Otherwise $1^{j[\alpha]}_{k[\beta]} \circ \epsilon^i_{j[\alpha]}=0$.
\\

\item[(iii)]  The composition of $ \epsilon^{j[\alpha]}_{k[\beta]} \circ  1^i_{j[\alpha]}$ is a morphism from $X_i$ to $X_k[\beta]$. If $-k < \beta \leq \min\{0,i-k\}$ holds, that is the condition of having a morphism of $k_\epsilon$-type between $X_i$ and $X_k[\beta]$, then $ \epsilon^{j[\alpha]}_{k[\beta]} \circ  1^i_{j[\alpha]}=\epsilon^i_{k[\beta]}$, as shown in the following diagram:

 $$\xymatrix@=15pt{
       &        & A\ar[r]\ar[d]^1 & \cdots\ar[r]\ar[d]^1 & \cdots\ar[r]\ar[d]^1 & \cdots\ar[r] &  A &\\
A\ar[r]  & \cdots\ar[r] & A\ar[r]\ar[d]^\epsilon & \cdots\ar[r]\ar[d]^0 & A\ar[d]^0      &        &  &\\
       &  & A\ar[r] & \cdots\ar[r] & \cdots\ar[r] & \cdots\ar[r]      & \cdots\ar[r] & A.
				 }$$\\

Otherwise $\epsilon^{j[\alpha]}_{k[\beta]} \circ  1^i_{j[\alpha]}=0$.
\\

\item[(iv)] The composition of two morphisms of $k_{\epsilon}$-type is always zero.
\end{enumerate}

As in Remark \ref{oss:NonPerfectIndecomposable}, the above results hold, with the same inequalities, also in $\dera$.

\section{Fully Faithful endofunctors of $\perfa$}

In this section we will deal with $k-$linear functors that commute with the shifts. Two simple examples of such type of functors are given by the shift $[n]$ and the push forward $\mathbf{R}f_*$ along a proper morphism $f$ of projective varieties.

Also, these two functors are exact and of Fourier-Mukai type; see \cite{HUY} for a deeper discussion.

For a more general analysis, in this section we will suppose $F$ to be fully faithful but we will not require the functor to be exact. Although this hypothesis is needed for the main result of this section, Corollary \ref{cor:AutoequivalenceToFM}, we keep this general setting in the section because sometimes it is useful to study functor that are not exact but satisfy the other properties. As an example, the reconstruction theorem from Fano and anti-Fano varieties works just considering equivalences that are not exact, see Remark 4.12 of \cite{HUY}.
Moreover we will manage to exploit this general context by giving an explicit example of an equivalence that is not exact, see Corollary \ref{cor:counterexample}.

\begin{prop} \label{prop:FunctorIsomorphicShift}
Let $F:\xymatrix{\perfa \ar[r] & \perfa}$ be a fully faithful functor. On the objects, $F$ is isomorphic to the shift functor $[n]$ for some integer $n$.
\end{prop}

\begin{proof}
$F$ commutes with the shifts, so we can focus on the image of an indecomposable object $X_i$  for any integer $i>0$. By Proposition \ref{prop:IndecomposableAndMaps}, $F$ sends indecomposable objects to indecomposable objects, so $F(X_i) \cong X_j[\alpha]$ for some integer $j>0$ and some $\alpha$. $F$ is also fully faithful, thus, for every integer $\beta$ one has:
$$\hom(X_i,X_i[\beta]) \cong \hom(F(X_i),F(X_i)[\beta]) \cong \hom(X_j[\alpha],X_j[\alpha+\beta]) \cong \hom(X_j,X_j[\beta]).$$

It follows from Corollary \ref{cor:gradedring} that $i=j$, and this proves that $F(X_i) \cong X_i[h_i]$ for some integer $h_i$. Actually $h_i$ does not depend on $i$. For every integer $\beta$ one has:
$$\hom(X_i,X_j[\beta]) \cong \hom(F(X_i),F(X_j)[\beta]) \cong \hom (X_i[h_i],X_j[h_j+\beta])\cong \hom(X_i,X_j[h_j-h_i+\beta]).$$

By a similar argument of Corollary \ref{cor:gradedring}, one concludes that $h_j$ has to be equal to $h_i$.
\end{proof}

\begin{cor} \label{cor:fullyfaithfulequivalence}
Every fully faithful functor $F:\xymatrix{\perfa \ar[r] & \perfa}$ is an equivalence.
\end{cor}
\begin{proof}
It is clear from Proposition \ref{prop:FunctorIsomorphicShift} that every fully faithful functor $\perfa \longrightarrow \perfa$ is also essentially surjective, hence it is an equivalence.
\end{proof}

With similar arguments, and by including the indecomposable objects $X_\infty$, Proposition \ref{prop:FunctorIsomorphicShift} and Corollary \ref{cor:fullyfaithfulequivalence} can be extended to a fully faithful functor $F:\dera \longrightarrow \dera$.

\begin{oss} \label{oss:ShiftReduction}
Due to Proposition \ref{prop:FunctorIsomorphicShift}, $F(X_i)$ is isomorphic to $X_i[h]$ for a fixed $h \in \Z$. Up to composition with a shift $[-h]$, we can assume that $F$ is isomorphic to the identity functor on the objects.
\end{oss}

We now want to study the action of $F$ on the morphisms between indecomposable elements.

\begin{prop} \label{prop:FunctorOnK2Morphism}
Consider a morphism $(a,b)$ as described in Remark \ref{oss:NonPerfectIndecomposable} from an indecomposable object $X_i$ to itself, that is
$$\centerline{
    \xymatrix{
 0\ar[r]\ar[d] & A\ar[r]\ar[d]^{a} & A\ar[r]\ar[d]^{a} \ar[r]\ar[d] &  \cdots\ar[r]\ar[d] & A\ar[r]\ar[d]^{a} & A\ar[r]\ar[d]^{a+ \epsilon\cdot b} & 0\ar[d] \\
 0\ar[r] & A\ar[r] & A\ar[r] &  \cdots\ar[r] & A\ar[r] & A\ar[r] & 0\\
  }
}$$

with $a$, $b\in k$. The action of the
functor on the morphism $(a,b)$ is given by an invertible matrix:
$$\left(
                                     \begin{array}{cc}
                                       1 & 0 \\
                                       0 & \delta_i \\
                                     \end{array}
                                   \right).
$$
with $\delta_i \in k$.
\end{prop}
\begin{proof}
Since $F$ is a functor, it preserves compositions and the identity. By imposing these two conditions to a generic $2\times2$ matrix with coefficients in $k$, the result is straightforward.

\end{proof}

This shows that if $(a,b)$ is a morphism of $k^2$-type from $X_i$ to itself, then $F$ acts only on its second component, which is the one generated by $k_{\epsilon}$. Hence the following definition makes sense.

\begin{defn}
For all $i,j\in \mathbb{N}$ and $\alpha\in \mathbb{Z}$ we define $k^i_{j[\alpha]}\in k$ such that:
\begin{itemize}
  \item if $(a,b)$ is a morphism of $k^2$-type from $X_i$ to $X_i$, then $F(a,b)=(a,k^i_{i[0]} b)$.
  \item if $(a,0)$ is a morphism of $k_1$-type from $X_i$ to $X_j[\alpha]$, then $F(a,0)=(k^i_{j[\alpha]} a,0)$.\item if $(0,b)$ is a morphism of $k_{\epsilon}$-type from $X_i$ to $X_{j}[\alpha]$, then $F(0,b)=(0,k^i_{j[\alpha]} b)$.
	\end{itemize}
	\end{defn}
	
	Note that the element $k^i_{i[0]}$ corresponds to $\delta_i$ in Proposition \ref{prop:FunctorOnK2Morphism}. The functor $F$ is fully faithful, hence all the coefficients $k^i_{j[\alpha]}$ are non zero.

\begin{prop} \label{prop:CoefficientNotDependingOnLength}
$k^i_{i[0]}$ does not depend on $i\in \mathbb{N} \smallsetminus \{0\}$.
\end{prop}
\begin{proof}
We prove that $k^{i}_{i[0]}=k^{1}_{1[0]}$ for $i>1$. Consider the following morphisms from $X_i$ to $X_1$ and from $X_1$ to $X_i$:
$$
\centerline{
    \xymatrix{
  X_i\ar[d]^{\epsilon^i_{1[0]}}:& 0\ar[r] & A\ar[r] & \cdots\ar[r] & A\ar[r]\ar[d]^\epsilon & 0\\
  X_1\ar[d]^{1^1_{i[0]}}: & 0\ar[r] & 0\ar[r] & 0\ar[r] & A\ar[r]\ar[d]^1 & 0\\
  X_i:& 0\ar[r] & A\ar[r] & \cdots\ar[r] & A\ar[r] & 0.\\
  }
}
$$

The functor $F$ sends $\epsilon^i_{1[0]}$ to $k^{i}_{1[0]}\epsilon^i_{1[0]}$ and $1^1_{i[0]}$ to $k^1_{i[0]}1^1_{i[0]}$; moreover, the composition $1^1_{i[0]} \circ \epsilon^i_{1[0]}=\epsilon^i_{i[0]}$ is a morphism between $X_i$ and $X_i$ and then it is sent by $F$ to $k^i_{i[0]}\epsilon^i_{i[0]}$.
As $F$ preserves compositions:
$$F(1^1_{i[0]} \circ \epsilon^i_{1[0]})=F(1^1_{i[0]})\circ F( \epsilon^i_{1[0]}),$$
which means $k^i_{i[0]}\epsilon^i_{i[0]} = k^1_{i[0]}k^{i}_{1[0]}\epsilon^i_{i[0]}$. It
follows $k^i_{i[0]} = k^1_{i[0]}k^{i}_{1[0]}$.
By composing these morphisms in the inverse order we get $\epsilon^i_{1[0]} \circ 1^1_{i[0]}=\epsilon^1_{1[0]}$, a morphism between $X_1$ and $X_1$. It is sent by $F$ to $(0,k^1_{1[0]}\epsilon^1_{1[0]})$.
Again, $F$ preserves compositions, hence
$k^1_{1[0]}=k^1_{i[0]}k^i_{1[0]}$, that is
$k^i_{i[0]}=k^1_{1[0]}$.
\end{proof}

\begin{prop} \label{prop:FunctorReduction}
Up to composing with a shift and a push forward along an automorphism of $\spec(A)$, the functor $F$ is isomorphic to a functor which is the identity on the objects and has coefficients $k^i_{i[0]}$ equal to $1$.
\end{prop}
\begin{proof}
Assume, as in Remark \ref{oss:ShiftReduction}, that $F$ is isomorphic to the identity on the objects. Moreover, it acts as the multiplication by $\mu := k^i_{i[0]}$ on the morphisms of $k^2$-type, which is constant by Proposition \ref{prop:CoefficientNotDependingOnLength}.
Now consider the map $\phi_{\mu}:\xymatrix{A\ar[r] & A}$ defined as follow:
$$a+\epsilon b\longmapsto a + \epsilon \mu b.$$
The induced push forward functor $(\phi_{\mu})_*$ on $\perfa$ is isomorphic to the identity on the objects and it acts as multiplication by
$\mu^{-1}$ on morphism of $k^2$-type. Up to isomorphisms of functors, the composition $(\phi_{\mu})_*\circ F$ is the identity on the objects and acts as the identity on morphisms of $k^2$-type.
\end{proof}

From now on, in view of Proposition \ref{prop:FunctorReduction}, we can assume that the functor $F$ satisfies the following condition:

\begin{enumerate}
\item[(C1)] $F$ is the identity on the objects of $\perfa$ and the coefficients $k^i_{i[0]}$ of $F$ are equal to $1$.
\end{enumerate}

\begin{lemma}
Let $k^i_{j[\alpha]}$ be the coefficient of a functor $F$ satisfying (C1). The following relations hold:

\begin{tabular}{lll}
	(R1)&  $k^j_{i[\alpha]}k^i_{j[-\alpha]}=1$                          & if $-i < \alpha \leq \min\{0,j-i\}$ or $\max\{0,j-i\} \leq \alpha < j$.\\
	(R2)&  $k^j_{i[\alpha]}=k^j_{i-1[\alpha]}k^{i-1}_{i[0]}$            & $0\leq \alpha< j \leq i$, $(i-j,\alpha) \neq (0,0), (1,0)$. \\
	(R3)&  $k^j_{i[\alpha]}=k^j_{i-1[0]}k^{i-1}_{i[\alpha]}$            & $j < i-1$ and $-i < \alpha \leq j-i$. \\
	(R4)&  $k^{i-1}_{i[\alpha]}=k^{i-1}_{i-1[\alpha]}k^{i-1}_{i[0]}$    & $1-i<\alpha<0$. \\
	(R5)&  $k^{i-1}_{i[2-i]}=k^{i-1}_{i-1[1]}k^{i-1}_{i[1-i]}$          & $i>2$. \\
\end{tabular}
\end{lemma}
\begin{proof}
(R1) For $i=j$ and $\alpha=0$ the statement is trivial. In the other cases note that, when the first inequality holds, $k^j_{i[\alpha]}$ is related to a morphism of $k_\epsilon$-type, $k^i_{j[-\alpha]}$ is related to a morphism of $k_1$-type and the composition is a non zero morphism of $k_\epsilon$-type between $X_j$ and $X_j$. When the second inequality holds, the types are swapped and the composition is still non zero. So we have:
$$k^j_{i[\alpha]} k^i_{j[-\alpha]}=k^j_{j[0]}=1.$$

(R2) The morphisms from $X_j$ to $X_{i}[\alpha]$, from $X_j$ to $X_{i-1}[\alpha]$ and from $X_{i-1}$ to $X_i[0]$ are of $k_1$-type, hence case (i) of Section 4 implies that the composition:
 $$1^j_{i[\alpha]}= 1^{i-1}_{i[0]} \circ 1^j_{i-1[\alpha]} $$
 is non zero.
 \\

 (R3) The morphism from $X_{j}$ to $X_{i-1}[0]$ is of $k_1$-type, the morphisms from $X_j$ to $X_i[\alpha]$ and the morphism from $X_{i-1}$ to $X_{i}[\alpha]$ are both of $k_\epsilon$-type, hence case C of Section 4 implies that the composition:
 	$$\epsilon^j_{i[\alpha]}=\epsilon^{i-1}_{i[\alpha]} \circ 1^j_{i-1[0]}$$
 	 is non zero.
 \\

 (R4) The morphism from $X_{i-1}$ to $X_{i}[0]$ is of $k_1$-type, the morphisms from $X_{i-1}$ to $X_{i}[\alpha]$ and the morphism from $X_{i-1}$ to $X_{i-1}[\alpha]$ are both of $k_\epsilon$-type, hence case (ii) of Section 4 implies that the composition:
 $$\epsilon^{i-1}_{i[\alpha]}=1^{i-1}_{i[0]}\circ \epsilon^{i-1}_{i-1[\alpha]}$$
  is non zero.
 \\

 (R5) The morphism from $X_{i-1}$ to $X_{i-1}[1]$ is of $k_1$-type, the morphisms from $X_{i-1}$ to $X_{i}[1-i]$ and the morphism from $X_{i-1}$ to $X_{i}[2-i]$ are both of $k_\epsilon$-type, hence case (iii) of Section 4 implies that the composition:
 $$\epsilon^{i-1}_{i[1-i]} \circ 1^{i-1}_{i-1[1]} = \epsilon^{i-1}_{i[2-i]}$$
  is non zero.

 \end{proof}

Given a set of objects $\E \subset \ob(\perfa)$ we denote by $\add{\E}$ the smallest full subcategory of $\perfa$ containing $\E$ and closed under shifts, finite direct sums and direct summand.

\begin{lemma} \label{lemma:NaturalIsomorphism}
Let $F$ be a functor satisfying (C1). The functor $F$ is isomorphic to a functor $F'$ such that the coefficients  $k'^{i-1}_{i[0]}$ of $F'$, are equal to $1$ for all $i>1$.
\end{lemma}
\begin{proof}
The isomorphism of functors between $F$ and $F'$ is given by the coefficients:
$$\phi_1=1 \quad \phi_i=\prod_{h=1}^{i-1}{(k^{h-1}_{h[0]})^{-1}} :\xymatrix{X_i \ar[r] & X_i}$$

The following diagram is commutative, and shows that $k'^{i-1}_{i[0]}=1$ concluding the proof of the lemma.

$$\xymatrix@=50pt{
 X_{i-1}\ar[r]^{\prod_{h=1}^{i-2}{(k^{h-1}_{h[0]})^{-1}}}\ar[d]_{f k^{i-1}_{i[0]}} & X_{i-1}\ar[d]^{f \cdot k'^{i-1}_{i[0]}}\\
 X_i\ar[r]^{\prod_{h=1}^{i-1}{(k^{h-1}_{h[0]})^{-1}}} & X_i.
 }$$\\
\end{proof}

Now we can choose a functor $F$ satisfying (C1) and, by Lemma \ref{lemma:NaturalIsomorphism}, the condition 

\begin{enumerate}
\item[(C2)] The coefficients $k^{i-1}_{i[0]}$ of $F$ are equal to $1$ for all $i>1$.
\end{enumerate}

\begin{teor} \label{teor:BehaviourFunctor}
Let $F$ be a functor satisfying (C1) and (C2).
The action of $F$ on the morphisms is completely determined by its coefficient $k^2_{1[1]}=\lambda$. In particular:
\begin{equation} \label{eqn:CoefficientFunctor}
k^i_{j[\alpha]}=\lambda^\alpha
\end{equation}
for $-j < \alpha \leq \min\{0,i-j\}$ or $\max\{0,i-j\} \leq \alpha < i$.
\end{teor}
\begin{proof}
We proceed by induction on the number of the indecomposable objects generating the subcategory $\add{X_1, \ldots, X_i}$.

On the subcategory $\add{X_1, X_2}$ we have:
$$k^1_{2[0]}k^2_{1[0]}\stackrel{(R1)}{=}1 \text{ and } k^2_{1[1]}k^1_{2[-1]}\stackrel{(R1)}{=}1,$$
$$k^2_{2[1]}\stackrel{(R2)}{=}k^2_{1[1]} k^1_{2[0]}\stackrel{(C2)}{=}\lambda,$$
$$k^2_{2[-1]}\stackrel{(R1)}{=}(k^2_{2[1]})^{-1}=\lambda^{-1}.$$
Note that, by Proposition \ref{prop:ClassificationMorphisms}, these equalities determine the behaviour of $F$ on all the coefficients and prove (\ref{eqn:CoefficientFunctor}) for the subcategory $\add{X_1, X_2}$. \\
Now assume that (\ref{eqn:CoefficientFunctor}) holds true for the subcategory $\add{X_1, \ldots, X_{i-1}}$ and prove it for the subcategory $\add{X_1, \ldots, X_{i}}$.
By assumption $k^i_{i[0]}=1$. By the description of the morphism of Proposition \ref{prop:ClassificationMorphisms}, it is clear that the following steps cover all the remaining coefficients of the functor on $\add{X_1, \ldots, X_{i}}$.

\begin{enumerate}
	\item[(i)] $k^j_{i[0]}$ for all $j<i$ (deducing the case of $k^i_{j[0]}$ by (R1)).
	\item[(ii)] $k^j_{i[\alpha]}$ for all $0<\alpha<j$, $j < i$ (deducing the case of $k^i_{j[-\alpha]}$ by (R1)).
	\item[(iii)] $k^j_{i[\alpha]}$ for all $-i<\alpha \leq j-i$, $j<i$ (deducing the case $k^i_{j[-\alpha]}$ by (R1)).
	\item[(iv)] $k^i_{i[\alpha]}$ for all $0<\alpha<i$ (deducing the case $k^i_{i[-\alpha]}$ by (R1)).
\end{enumerate}

As for the proof:

\begin{enumerate}
	\item[(i)] If $j=i-1$ one obtains $k^{i-1}_{i[0]}=1$ by (C2). \\
	For $j< i-1$, by induction $k^j_{i-1[0]}=1$. We have:
	
	$$k^j_{i[0]}\stackrel{(R2)}{=}k^j_{i-1[0]}k^{i-1}_{i[0]}\stackrel{(C1)}{=}1.$$
	
	\item[(ii)] By induction $k^j_{i-1[\alpha]}=\lambda^{\alpha}$. The claim is true because
	$$k^j_{i[\alpha]}\stackrel{(R2)}{=}k^j_{i-1[\alpha]} k^{i-1}_{i[0]}=\lambda^{\alpha}.$$
	
	\item[(iii)] If $j \neq i-1$, by induction $k^j_{i-1[0]}=1$, then:
	$$k^j_{i[\alpha]}\stackrel{(R3)}{=}k^j_{i-1[0]} k^{i-1}_{i[\alpha]}=k^{i-1}_{i[\alpha]}.$$
	Therefore we have to prove the claim only for $k^{i-1}_{i[\alpha]}$.
	In this case $1-i \leq \alpha < 0$.
	\begin{itemize}
	\item If $1-i < \alpha < 0$, by induction $k^{i-1}_{i-1[\alpha]}=\lambda^\alpha$, then:
	$$k^{i-1}_{i[\alpha]}\stackrel{(R4)}{=}k^{i-1}_{i-1[\alpha]} k^{i-1}_{i[0]}\stackrel{(C2)}{=}\lambda^{\alpha}$$
	\item If $\alpha = 1-i$, it is sufficient to note that $k^{i-1}_{i[\alpha+1]}=k^{i-1}_{i[2-i]}$ belongs to the previous case and by induction $k^{i-1}_{i-1[1]}=\lambda$. Hence:
	$$k^{i-1}_{i[1-i]} \stackrel{(R5)}{=} k^{i-1}_{i[2-i]} (k^{i-1}_{i-1[1]})^{-1}=\lambda^{\alpha+1} \lambda^{-1}=\lambda^{\alpha}.$$
	\end{itemize}
	
	\item[(iv)] We have:
	$$k^i_{i[\alpha]}\stackrel{(R2)}{=}k^i_{i-1[\alpha]} k^{i-1}_{i[0]}=k^{i}_{i-1[\alpha]}\stackrel{(iii)}{=}\lambda^{\alpha}.$$
\end{enumerate}

So the claim is true.
\end{proof}

\begin{cor} \label{cor:ExactToIdentity}
Let $F$ be a functor satisfying (C1) and (C2).
If $F$ is exact, then it is isomorphic to the identity functor.
\end{cor}
\begin{proof}
It suffices to show that, if $F$ is exact, then $\lambda=k^2_{1[1]}=1$.
Consider the following distinguished triangle:
$$\xymatrix{
X_1 \ar[r]^{\epsilon^1_{1[0]}} & X_1 \ar[r]^i & C(\epsilon^1_{1[0]}) \ar[r]^p & X_1[1],
}$$
since the cone $C(\epsilon^1_{1[0]})$ on the morphism $\epsilon^1_{1[0]}$ is isomorphic to $X_2$, the triangle becomes:
\begin{equation} \label{eqn:DistinguishedTriangleStart}
\xymatrix{
X_1 \ar[r]^{\epsilon^1_{1[0]}} & X_1 \ar[r]^{1^1_{2[0]}} & X_2 \ar[r]^{1^2_{1[1]}} & X_1[1].
}
\end{equation}
Now $F$ sends the previous triangle in to the following one:
\begin{equation} \label{eqn:DistinguishedTriangle}
\xymatrix{X_1 \ar[r]^{\epsilon^1_{1[0]}} & X_1 \ar[r]^{1^1_{2[0]}} & X_2 \ar[r]^{1^2_{1[1]} \lambda} & X_1[1].}
\end{equation}
Since $F$ is exact, the triangle (\ref{eqn:DistinguishedTriangle}) is distinguished, hence it is isomorphic to the distinguished triangle (\ref{eqn:DistinguishedTriangleStart}). So we have
$$\xymatrix{
X_1\ar[r]^{\epsilon^1_{1[0]}}\ar[d]^{\id} & X_1\ar[r]^{1^1_{2[0]}}\ar[d]^{\id} & X_2\ar[r]^{\lambda 1^2_{1[1]}}\ar[d]^{a+\epsilon b} & X_1[1]\ar[d]^{\id}\\
X_1\ar[r]^{\epsilon^1_{1[0]}} & X_1\ar[r]^{1^1_{2[0]}} & X_2\ar[r]^{1^2_{1[1]}} & X_1[1].
 }$$
 The diagram is commutative up to homotopy, hence:
 $$\left\{\begin{array}{l}
 a=\lambda\\
 a=1\\
 \end{array}\right.$$

Thus $\lambda=1$ and, by Theorem \ref{teor:BehaviourFunctor}, the functor $F$ is the identity.
\end{proof}

\begin{cor} \label{cor:AutoequivalenceToFM}
Every exact autoequivalence of $\perfa$ is of Fourier-Mukai type.
\end{cor}
\begin{proof}
See \cite{HUY}, Proposition 5.10, for the proof that the composition of Fourier-Mukai functor is again of Fourier-Mukai type.
$F$ is the identity up to shifts and push forwards functors, which are both of Fourier-Mukai type. Hence $F$ itself is a Fourier-Mukai functor.
\end{proof}

\begin{cor} \label{cor:counterexample}
If $k \neq \Z_2$, then there exist an autoequivalence of $\perfa$ that is not exact.
\end{cor}
\begin{proof}
Choose the coefficient $k^1_{2[1]} \neq 0, 1$, set all the coefficients as described in the Theorem \ref{teor:BehaviourFunctor}. The functor $F$ is well defined since all the compositions are well posed:
$$k^i_{j[\alpha]} k^{j[\alpha]}_{l[\beta]} = k^i_{j[\alpha]} k^{j}_{l[\beta-\alpha]}= \lambda^{\alpha + \beta - \alpha} = \lambda^{\beta} = k^i_{l[\beta]}.$$
By Corollary \ref{cor:ExactToIdentity}, $F$ is not exact.
\end{proof}

\section{Main Theorem}
The proof of Theorem \ref{teor:LuntsOrlovT} uses strongly the language of DG categories. A survey of the subject can be found in \cite{DGKeller}. Here are the definitions used in the proof.
\begin{defn}A DG category is a k-linear category $\mathcal{A}$ such that:
\begin{itemize}
  \item $\mathrm{Hom}(X,Y)$ is a $\mathbb{Z}$-graded $k$-module for every $X,Y\in \mathrm{Ob}(\mathcal{A})$.
  \item There is a differential $d:\mathrm{Hom}(X,Y)\longrightarrow \mathrm{Hom}(X,Y)$ of degree one, such that for every $X,Y,Z\in \mathrm{Ob}(\mathcal{A})$ the composition $\mathrm{Hom}(X,Y)\otimes \mathrm{Hom}(Y,Z)\longrightarrow \mathrm{Hom}(X,Z)$ is a morphism of DG $k$-modules.

\end{itemize}
\end{defn}
A \emph{DG functor} $\mathcal{F}:\mathcal{A}\longrightarrow \mathcal{B}$ between two DG categories is given by a map on the objects $\mathcal{F}:\mathrm{Ob}(\mathcal{A})\longrightarrow \mathrm{Ob}(\mathcal{B})$ and maps on the spaces of morphisms:
$$\mathcal{F}(X,Y):\mathrm{Hom}_{\mathcal{A}}(X,Y)\longrightarrow \mathrm{Hom}_{\mathcal{B}}(\mathcal{F}(X),\mathcal{F}(Y))$$
which are morphisms of DG $k$-modules and are compatible with the compositions and the units.\\
Given a DG category $\mathcal{A}$, we denote by $H^0(\mathcal{A})$ the \emph{homotopy category} associated to $\mathcal{A}$, which has the same objects of the DG category $\mathcal{A}$ and its morphisms are defined by taking the zeroth cohomology $H^0(\mathrm{Hom}_{\mathcal{A}}(X,Y))$.

\begin{defn}
A DG functor $\mathcal{F}:\mathcal{A}\longrightarrow \mathcal{B}$ is called a quasi-equivalence if $\mathcal{F}(X,Y)$ is a quasi-isomorphism for all objects $X,Y\in \mathcal{A}$ and the induced functor $H^0(\mathcal{F}):H^0(\mathcal{A})\longrightarrow H^0(\mathcal{B})$ is an equivalence. We say that two objects $a,b \in \mathcal{A}$ are homotopy equivalent if they are isomorphic in $H^0(\mathcal{A})$.
\end{defn}

\begin{defn}
Let $\mathrm{dgMod}$-$k$ be the DG category of  DG $k$-modules. Given a small DG category $\mathcal{A}$, every DG functor: $$\mathcal{M}:\mathcal{A}^{op}\longrightarrow \mathrm{dgMod}\text{-}k$$ is called a right DG $\mathcal{A}$-module.
\end{defn}

We denote by $\mathrm{dgMod}$-$\mathcal{A}$ the DG category of right DG $\mathcal{A}$-modules. Let $\mathrm{Ac}(\mathcal{A})$ be the subcategory of $\mathrm{dgMod}$-$\mathcal{A}$ consisting of all acyclic DG modules. \begin{defn}The derived category $\der(\mathcal{A})$ is the Verdier quotient between the homotopy category associated with $\mathrm{dgMod}$-$\mathcal{A}$ and the subcategory of acyclic DG modules:
$$\der(\mathcal{A}):=\frac{H^0(\mathrm{Mod}\text{-}\mathcal{A})}{H^0(\mathrm{Ac}(\mathcal{A}))}.$$
\end{defn}

Every object $X\in \mathcal{A}$ defines a \emph{representable} DG module:
$$h^X(-):=\mathrm{Hom}(-,X).$$The functor $h^{\bullet}$ is called the \emph{Yoneda functor}, and it is fully faithful.
\begin{defn}A DG $\mathcal{A}$-module $\mathcal{M}$ is called free if it isomorphic to a direct sum of shift of representable DG modules of the form $h^X[n]$, where $X\in \mathcal{A}$, $n\in \mathbb{Z}$.\end{defn}
\begin{defn}A DG $\mathcal{A}$-module $\mathcal{P}$ is called semi-free if it has a filtration: $$0=\phi_0\subset\phi_1\subset\phi_2\subset \ldots = \mathcal{P}$$ such that each quotient $\phi_i/\phi_{i-1}$ is free.\end{defn}

If $\phi_m=\mathcal{P}$ for some $m$ and $\phi_i/\phi_{i-1}$ is a finite direct sum of DG modules of the form $h^Y[n]$, then we call $\mathcal{P}$ a \emph{finitely generated semi-free} DG module. Denote by $\mathcal{SF}(\mathcal{A})$ the full DG subcategory of semi-free DG modules.

\begin{defn}
Given a small DG category $\mathcal{A}$ we denote by $Perf(\mathcal{A})$ the DG category of perfect DG modules, that is the full DG subcategory of $\mathcal{SF}(\mathcal{A})$ consisting of all DG modules which are homotopy equivalent to a direct summand of a finitely generated semi-free DG module.
\end{defn}

Recall that, given two DG categories $\mathcal{A}$ and $\mathcal{B}$, their tensor product $\mathcal{A}\otimes \mathcal{B}$ is again a DG category. See \cite{BLL} for references.

Let $\mathcal{A}$ and $\mathcal{B}$ be two DG categories, an \emph{$\mathcal{A}$-$\mathcal{B}$-bimodule} is a DG $\mathcal{A}^{op}\otimes \mathcal{B}$-module. A \emph{quasi-functor} from $\mathcal{A}$ to $\mathcal{B}$ is a $\mathcal{A}$-$\mathcal{B}$-bimodule $X\in \der(\mathcal{A}^{op}\otimes \mathcal{B})$ such that the tensor functor: $$(-)\otimes_{\mathcal{A}}X:\der(\mathcal{A})\longrightarrow \der(\mathcal{B})$$takes every representable DG $\mathcal{A}$-module to an object which is isomorphic to a representable DG $\mathcal{B}$-module.

\begin{defn}
Let $\T$ be a triangulated category. An enhancement of $\T$ is a pair $(B,\epsilon)$, where $\mathcal{B}$ is a pretriangulated DG category and $\epsilon$ : $H^0(\mathcal{B})\longrightarrow \T$ is an equivalence of triangulated categories.
\end{defn}

In the following we give a slight different version of \cite{LO}, Corollary 9.13, which extends the results in \ref{cor:AutoequivalenceToFM}.

\begin{teor} \label{teor:LuntsOrlovT}
Let $Y$ be a quasi-compact and separated scheme. Let:
$$F:\xymatrix{\perfa \ar[r] & \der(\qcoh Y)}$$
be a fully faithful functor. Then
there is an object $\mathcal{E}\in
\der(\qcoh (\spec A \times Y))$ such that:
$${\Phi_{\mathcal{E}}}|_{\perfa}\cong F.$$Furthermore, if $Y$ is noetherian and $F$ sends $\perfa$ to $\der^b(Y)$, then
$$\mathcal{E}\in \der^b(\spec A \times Y).$$
\end{teor}

\begin{proof}
We know by (\cite{LO}) that there exist enhancements of the derived categories $\der(\qcoh (\spec A))$ and
$\der(\qcoh (Y))$, we call them $\mathcal{D}_{dg}(\qcoh (\spec A))$ and $\mathcal{D}_{dg}(\qcoh (Y))$ respectively. Also, by \cite{LO} Proposition $1.17$, these enhancements are quasi equivalent to the DG categories $\mathcal{SF}(Perf(A))$ and $\mathcal{SF}(Perf(Y))$. Denote by:
$$\phi_A:\xymatrix{\mathcal{D}_{dg}(\qcoh A)\ar[r] & \mathcal{SF}(Perf(A))}$$
$$\phi_Y:\xymatrix{\mathcal{D}_{dg}(\qcoh Y)\ar[r] & \mathcal{SF}(Perf(Y))}$$
the corresponding quasi-functors. The functor $F$ induces an equivalence:
$$\tilde{F}:\xymatrix{\perfa\ar[r]^\sim & H^0(\mathcal{C})}$$
where $\mathcal{C}$ is the full DG subcategory in $\mathcal{SF}(Perf(Y))$ consisting of all objects in the essential image of $H^0(\phi_Y)\circ F$.
By \cite{LO}, Theorem $6.4$, there is a
quasi-equivalence:
$$\mathcal{F}:\xymatrix{Perf(A)\ar[r] & \mathcal{C}}$$ which induces a
quasi-equivalence:
$$\mathcal{F}^*:\xymatrix{\mathcal{SF}(Perf(A))\ar[r] & \mathcal{SF}(\mathcal{C})}.$$
Let $\mathcal{D}\subset
\mathcal{SF}(Perf(Y))$ be a DG subcategory that contains $Perf(Y)$
and $\mathcal{C}$. Denote by $\mathcal{J}: \xymatrix{\mathcal{C}\ar[r] & \mathcal{D} }$ and
$\mathcal{I}:\xymatrix{Perf(Y)\ar[r] & \mathcal{D}}$ the respective
embeddings. Let $\mathcal{H}:= \phi_Y^{-1}\circ \mathcal{I_*}\circ
\mathcal{J^*}\circ \mathcal{F^*}\circ \phi_A:
\xymatrix{\mathcal{D}_{dg}(\qcoh A)\ar[r] &
\mathcal{D}_{dg}(\qcoh Y)}$ be the functor that makes the following diagram be commutative:
$$\centerline{
\xymatrix{
 \mathcal{D}_{dg}(\qcoh A)\ar[rrr]^{\mathcal{H}}\ar[d]^{\phi_A} & &  & \mathcal{D}_{dg}(\qcoh Y)\ar[d]^{\phi_Y}\\
 \mathcal{SF}(Perf(A))\ar[r]^-{\mathcal{F}^*} & \mathcal{SF}(\mathcal{C})\ar[r]^{\mathcal{J}^*} & \mathcal{SF}(\mathcal{D})\ar[r]^-{\mathcal{I}_*} & \mathcal{SF}(Perf(Y))
 }}
$$
Notice that $H^0(\mathcal{H})$ commutes with direct sums, hence (\cite{LO}, Theorem $9.10$) the functor $H^0(\mathcal{H})$ is isomorphic to $\Phi_{\mathcal{E}}$ with $\mathcal{E}\in \der(\qcoh (\spec A \times Y))$.\\
As observed in the proof of \cite{LO}, the restriction of $\mathcal{I}_*\circ \mathcal{J}^*$ on $\mathcal{C}$ is isomorphic to the inclusion $\xymatrix{\mathcal{C}\ar[r] & \mathcal{SF}(Perf(Y))}$, hence the restriction ${\Phi_{\mathcal{E}}}|_{\perfa}$ is fully faithful.

Let $\mathfrak{A}$ be the full subcategory of $\perf(A)$ whose object is only $A$, and let $j:\mathfrak{A} \rightarrow \perfa$ be the natural embedding.

Define:
$$G:=H^0(\mathcal{F})^{-1}\circ
\tilde{F}:\xymatrix{\perfa\ar[r] & \perfa}$$
By \cite{LO}, Theorem $6.4$, there is an isomorphism of functors:
$$\xymatrix{j \ar[r]^\sim & G\circ j}$$on the category $\mathfrak{A}$. Hence, by Corollary \ref{cor:ExactToIdentity}, the functor $G$ is the identity on the whole $\perfa$. Therefore, the functors
$H^{0}(\mathcal{F})$ and $\tilde{F}$ are isomorphic, that is:
$$(H^0(\phi_Y)\circ H^0(\mathcal{H}))|_{\perf(A)} \cong (H^0(\phi_Y)\circ F)\Rightarrow {\Phi _{\mathcal{E}}}|_{\perfa}\cong F.$$

Finally if $Y$ is noetherian and $F$ sends $\perfa$ to $\der^b(Y)$, then \cite{LO}, Corollary $9.13$, implies: $$\mathcal{E}\in
\der^b(\mathrm{Coh}(\spec A \times Y)).$$

\end{proof}

\begin{cor}
Let $Y$ be a quasi-compact and separated scheme. Let: $$F:\xymatrix{\der^b(\spec A) \ar[r] & \der(\qcoh Y)}$$ be a fully faithful functor that commutes with homotopy colimits. Then there is an object $\mathcal{E}\in \der(\qcoh (X\times Y))$ such that:
$${\Phi_{\mathcal{E}}}|_{\dera}\cong F.$$
\end{cor}

\begin{proof}
Corollary 9.14 in \cite{LO} shows a similar result: if $X$ is a projective scheme such that $T_0(\mathcal{O}_X)=0$ and $Y$ is a quasi-compact separated scheme, then for every fully faithful functor that commutes with homotopy colimits:
$$F:\der^b(X)\longrightarrow \der(\qcoh (Y))$$
there is an object $\mathcal{E}\in \der(\qcoh (X \times Y))$ such that:
$${\Phi_{\mathcal{E}}}|_{\der^b(X)}\cong F.$$
The authors assume $T_0(\mathcal{O}_X)=0$ in order to prove that the restriction of the functor $F$ to the subcategory of perfect complexes $\mathbf{Perf}(X)$ is of Fourier-Mukai type. In our case we have actually $T_0(\mathcal{O}_A)\neq 0$, but we have already shown in Theorem \ref{teor:LuntsOrlovT} that the restriction of $F$ to $\mathbf{Perf}(A)$ is a Fourier-Mukai functor. Hence we do not need this hypothesis and the proof follows as in Corollary 9.14, \cite{LO}.

\end{proof}

\section{$t$-structures on $\dera$}
This section is devoted to the study of all the possible $t$-structures on $\dera$. 

\begin{defn} \label{defn:TStructure}
Let $\T$ be a triangulated category. A $t$-structure on $\T$ is given by a full additive subcategory $\F$ such that:
\begin{itemize}
\item $\F[1] \subset \F$
\item For all objects $E$ in $\T$, there exists a distinguished triangle:
$$F \rightarrow E \rightarrow G$$
where $F \in \F$ and $G \in \F^\perp$.
\end{itemize}
The heart of a $t$-structure is the subcategory $\A:=\F \cap \F^\perp[1]$.
\end{defn}

All the $t$-structures we use in the following are bounded, that means every objects $E \in \T$ belongs to $\F[i] \cap \F^\perp[j]$ for some $i$ and $j$. The trivial $t$-structures on $\T$ are given by $\F=0$ or $\F=\T$.\\ It seems natural to wonder if a specific heart identifies a unique $t$-structure. An answer to this question has been given in Lemma 3.2 in \cite{SCOTC}, which states:
\begin{prop} \label{prop:PropertyHeart}
Let $\A$ be a full additive subcategory of a triangulated category $\T$. $\A$ is the heart of a bounded $t$-structure if and only if the following properties hold:
\begin{enumerate}
\item For every objects $A$ and $B$ of $\A$ and for every integer $h_1 > h_2$, then $\hom(A[h_1],B[h_2])=0$.
\item For every object $E$ of $\T$, there exist a finite sequence of integers $h_1 > h_2 > \ldots >h_n$ and a collection of distinguished triangles:
$$\xymatrix@C=1em{
0\ar[rr] && E_1\ar[rr]\ar[dl] && E_2\ar[r]\ar[dl] & \cdots\ar[r] & E_{n-1}\ar[rr] && E_n=E\ar[dl]\\
& A_1\ar@{-->}[ul] && A_2\ar@{-->}[ul] &&&& A_n\ar@{-->}[ul]
}$$
with $A_j \in \A[h_j]$ for all $j$.
\end{enumerate}
The subcategory $\F$ is then generated by extension of the subcategories $\A[h]$, $h \geq 0$.
\end{prop}

The following remark gives some well known properties about $t$-structures and hearts. See \cite{FP}, \cite{INTRODUCTION} and \cite{SCOTC} for more details.

\begin{oss} \label{oss:PropertyOfHeart}
Given a heart of a bounded $t$-structure, the filtration provided by the property $(2)$ of Proposition \ref{prop:PropertyHeart} has the following properties:
\begin{enumerate}
\item The heart of a $t$-structure is an abelian subcategory closed under extensions. 
\item The filtration is unique up to isomorphisms. In particular, the shifts $k_j$ are fixed.
\item The filtration of the object $X[h]$ can be deduced from the filtration of $X$.
\item The filtration of the object $X \oplus Y$ can be deduced from the filtrations of $X$ and of $Y$.
\end{enumerate}
\end{oss}

The first case we are interested in is the case of $\perfa$. Holm, J{\o}rgensen and Yang proved, in the context of spherical objects (\cite{SPARSENESS}), that all the $t$-structures on $\perfa$ are trivial. The proof of this fact follows easily by a direct calculation: the only possible candidate for being an heart is the subcategory $\add{X_1[h]}$ (See Proposition \ref{prop:UniqueStructure} for a similar proof). However, such a subcategory does not satisfy the property $(2)$ of Proposition \ref{prop:PropertyHeart}.

Let us turn to analyze the case of the category $\dera$. In fact, it is generated, as triangulated category, by the indecomposable object $X_\infty$. This case is not covered by \cite{SPARSENESS}, since $\dera$ is not generated by any spherical object. 

\begin{oss}
The subcategory $\F=\{X \in \dera \st H^i(X)=0$ for every $i \geq 0\}$ is the standard $t$-structure on $\dera$. Its heart is the subcategory $A-\operatorname{mod}_{\operatorname{fg}}[1]=\add{X_1, X_\infty}$.
\end{oss}

\begin{prop} \label{prop:UniqueStructure}
Up to shift, the unique $t$-structure on $\dera$ is the standard one.
\end{prop}
\begin{proof}
We look for all possible hearts satisfying the two properties of Proposition \ref{prop:PropertyHeart}.
Since the heart $\A$ is abelian, it is sufficient to check which indecomposable objects does $\A$ contain.
Thanks to the first part of Proposition \ref{prop:PropertyHeart} it is easy to verify that, up to shifts, the only admissible candidates for hearts are $\A=\add{X_1}$, $\A=\add{X_\infty}$ and $\A=\add{X_1,X_\infty}$. The first case is not possible, since $X_1$ does not generate the whole category $\dera$. The distinguished triangle:
$$X_\infty \xrightarrow{\epsilon} X_1 \xrightarrow{1} X_\infty$$
is an extension of $X_1$ by elements of $\add{X_\infty}$, and so if $X_\infty$ is an element of $\A$, then $X_1$ is such. It follows that the unique possibility is $\A=\add{X_1,X_\infty}$.
\end{proof}

It could be interesting to look at the explicit construction of the filtration for the objects of $\dera$. The first step is to write the filtration of the indecomposable objects of $\dera$. $X_1$, $X_\infty$ and all the other elements of the heart have the filtration provided by the distinguished triangle $0 \rightarrow \square \rightarrow \square$. As for other indecomposables, by taking the cone one has the following exact triangle, for $1<i < \infty$:
$$X_\infty \xrightarrow{\epsilon} X_i[-i+1] \rightarrow X_\infty[-i+1].$$
The filtration of the indecomposable object $X_i[-i+1]$ is the following:
$$\xymatrix@C=1em{
0\ar[rr] && X_\infty\ar[rr]^\epsilon\ar[dl] && X_i[-i+1]\ar[dl]\\
& X_\infty\ar@{-->}[ul] && X_\infty[-i+1]\ar@{-->}[ul]
}$$
By Remark \ref{oss:PropertyOfHeart}, the filtration of other indecomposable objects can be obtained by shifting these ones above. Moreover, the filtration of every object $X$ of $\dera$  can be constructed by taking direct sums of the filtration of indecomposable objects that generate $X$.

\section{Stability conditions on $\dera$}
In this section we will describe the space $\stab(\dera)$ of stability conditions on $\dera$. The proofs of this section are inspired by \cite{COSTAB}, where J{\o}rgensen and Pauksztello describe the space of co-stability conditions on $\perfa$. We will give a brief description of stability conditions following \cite{SCOTC} and \cite{INTRODUCTION}.
\begin{defn}
Let $\A$ be an abelian category. A stability function on $\A$ is a group homomorphism $Z:K(\A) \rightarrow \C$ such that for every non zero object $E$ of $\A$, the number $Z(E)$ belongs to:
$$H=\left\{z \in \C \st z=\rho \exp(i \pi \phi), \rho \geq 0, 0 < \phi \leq 1 \right\}.$$
The phase of $E \in \A$ is the real number $(1/\pi) \arg{\left(Z(E)\right)} \in (0,1]$.\\
A non zero object $E \in \A$ is called semi-stable if every non zero sub-object $S \hookrightarrow E$ has the phase less or equal to the phase of $E$.

\end{defn}

Thanks to the results of previous section, we know that all the $t$-structures on $\dera$ are given by shifts of the standard one. In particular all the possible heart are $\A_h=\add{X_1[h],X_\infty[h]}$. The exact sequence:
$$0 \rightarrow X_\infty \xrightarrow{\epsilon} X_1 \xrightarrow{1} X_\infty \rightarrow 0$$
gives a relation in the Grothendieck group $\left[X_1[h]\right]=2\left[X_\infty[h]\right]$. It follows then that the Grothendieck group is the free abelian group generated by $X_\infty[h]$. All objects of the hearts $\A_h$ are semi-stable.

In order to give the stability function, it suffices to choose a vector $v$ in $H$ as the image of $X_\infty[h]$. 

\begin{defn} \label{defn:StabilityCondition} 
A stability condition $\sigma=(Z,\Psharp)$ on a triangulated category $\T$ is provided by a group homeomorphism $Z:K(\T) \rightarrow \C$ called \textit{central charge} and subcategories $\Psharp(\phi)$ of $\T$, indexed by $\phi \in \R$ such that:
\begin{enumerate}
	\item For every object $0 \neq E \in \Psharp(\phi)$, $Z(E)$ has phase $\phi$.
	\item $\Psharp(\phi+1)=\Psharp(\phi)[1]$ for all $\phi \in \R$.
	\item If $\phi_1 > \phi_2$, $E_i \in \Psharp(\phi_i)$, then $\hom_\T(E_1,E_2)=0$.
	\item Any non zero object $E$ admits a Harder-Narasimhan filtration, that is a finite number of inclusion
	$$0=E_0 \hookrightarrow E_1 \hookrightarrow \cdots \hookrightarrow E_{n-1} \hookrightarrow E_n = E$$
	such that $F_j=\cone(E_{j-1} \hookrightarrow E_j)$ are semistable object of  with phase
	$$\phi(F_1)> \cdots > \phi(F_{n-1}) >\phi(F_n)$$
\end{enumerate}
\end{defn}

$\stab(\T)$ denotes the set of stability conditions which are locally finite. Note that, if $K(\T)$ is discrete, as in the case we are dealing with, all the stability conditions are locally finite. Bridgeland proved that this space has a natural topology defined by a generalized metric. $\stab(\T)$ endowed with this topology, turns out to be a complex manifold. If the Grothendieck group is finitely generated, as in our case, this manifold is of finite dimension.

\begin{prop} \label{prop:GivingStability}
A stability condition on $\dera$ can be given by an integer $h$ and a vector $v \in H$.
\end{prop}
\begin{proof}
By Proposition 5.3 of \cite{SCOTC}, it is sufficient to provide a heart of a bounded $t$-structure and a stability function. The Harder-Narashiman property required in the theorem is assured since the heart is artinian and noetherian. The integer $h$ specifies the heart $\A_h$ as described above, and $v$ describes the stability function $Z(X_\infty[h])=v$. 

The Grothendieck group $K(\A_h)$ is, for every $h$, isomorphic to the Grothendieck group of the whole category $\dera$ required in \ref{defn:StabilityCondition} (see \cite{GROT} for details).
The data $(h,v)$ correspond to the stability condition $\sigma=(Z,\Psharp)$ where the group homeomorphism $Z$ is given by the stability function as observed above. Let $\phi$ be the phase of $v$; $\Psharp$ is given by: 
$$\Psharp(\phi)=\add{X_1[h], X_\infty[h]}$$
and it is zero for all the other $\phi \in (0,1]$, these data extend to all $\phi \in \R$ by the property $(2)$ of Definition \ref{defn:StabilityCondition}.
\end{proof}

There are two group actions on the space $\stab(\T)$ (See \cite{SCOTC}, Lemma 8.2): a right action of $\widetilde{\gl^+}(2,\R)$ and a left action by isometries of the group of auto-equivalences of the category $\D$.

\begin{oss}
The elements of $\widetilde{\gl^+}(2,\R)$ (the universal covering of $\gl^+(2,\R)$)  are pairs $(G,f)$ where $G \in\gl^+(2,\R)$ and $f: \R \rightarrow \R$ such that:
\begin{itemize}
\item $f$ is an increasing map with $f(x+1)=f(x)+1 \forall x \in R$.
\item $\frac{G\exp{i\pi\phi}}{|G\exp{i\pi\phi}|}=\exp{i\pi f(\phi)}$.
\end{itemize}
\end{oss}

Let $S$ be the subgroup of $\widetilde{\gl^+}(2,\R)$ generated by rotations $(\exp{i\pi\theta},f(x)=x+\theta)$, $\theta \in \R$ and scalings $(k,f(x)=x)$, $k \in \R^+$. We have:
$$S=\{(k\exp{i\pi\theta},\quad f(x)=x+\theta)\quad\forall \theta \in \R,\quad k \in \R^+\}.$$
The action on $\stab(\T)$ is given by:
$$(G,f) \circlearrowright (Z,\Psharp(\phi)) = (G^{-1} \circ Z, \Psharp(f(\phi)).$$

\begin{lemma} \label{lemma:GroupAction}
The action of $S$ on $\stab(\dera)$ is free and transitive, hence $\stab(\dera) \cong S$.
\end{lemma}
\begin{proof}
Let $(Z,\Psharp)$ be the given stability condition, as in Proposition \ref{prop:GivingStability}, by the pair $(h,v)$. Note that $v=|v|\exp{i \pi \phi_v}$. $\Psharp(\phi_v)=\add{X_1[h], X_\infty[h]}$ thus $\Psharp(\phi_v-h)=\add{X_1, X_\infty}$. Let $\theta=-h-\phi_v$ and $k=(|v|)^{-1}$. By the element $(k\exp{i\pi\theta},f(x)=x+\theta)$ of $S$ one can see that $(Z,\Psharp)$ belongs to the same orbit of the stability condition $(0,-1)$. Then the action is transitive. Moreover it is straightforward to verify that the action is also free.
\end{proof}

\begin{teor} \label{thm:Classification}
$\stab(\dera)$ is isomorphic to $\C$, the universal covering of $\C^*$.
\end{teor}
\begin{proof}
$\C$ is the universal covering of $\C^*$ by standard arguments.
By Lemma \ref{lemma:GroupAction}, it is sufficient to verify the claim for the subgroup $S$. $S$ is the universal covering of the subgroup of $\gl^+(2,\R)$ given by $\{k \exp(i \pi \theta), k \in \R^+, \theta \in \R \}$, which is isomorphic to $\C^*$.
\end{proof}

In fact there is a sort of dual notion of the $t$-structures, called co-$t$-structure. It was introduced in \cite{BONDARKO} and \cite{P} by Bondarko and Pauksztello. Moreover, in \cite{SPARSENESS} is shown that on $\perfa$ there are non-trivial co-$t$-structures. 

\begin{defn} \label{defn:CoTStructure}
Let $\T$ be a triangulated category. A co-$t$-structure on $\T$ is given by a full additive subcategory $\F$ such that:
\begin{itemize}
\item $\F[-1] \subset \F$
\item For all objects $E$ in $\T$, there exists a distinguished triangle:
$$F \rightarrow E \rightarrow G$$
where $F \in \F$ and $G \in \F^\perp$.
\end{itemize}
The co-heart of a co-$t$-structure is the subcategory $\A:=\F \cap \F^\perp[-1]$.
\end{defn}

There are important differences between $t$-structures and co-$t$-structures. One example is provided by the properties $(1)$ and $(2)$ of Remark \ref{oss:PropertyOfHeart}: there are examples of co-hearts of a co-$t$-structures that are not abelian and in general the filtration $(2)$ is not unique. Proposition 1.3.3 of \cite{BONDARKO} makes clear that the proof of Proposition \ref{prop:UniqueStructure} still works in the context of co-$t$-structures. The notion of co-stability conditions is also rather similar to the one of stability condition given in Definition \ref{defn:StabilityCondition}. Note that the inequality of part $(3)$ of definition is reversed
\begin{enumerate}
	\item[(co-3)] If $\phi_1 < \phi_2$, $E_i \in \Psharp(\phi_i)$, then $\hom_\T(E_1,E_2)=0$. 
\end{enumerate}
as are reversed the inequalities involving the shifts in the Harder-Narasimhan filtration
\begin{enumerate}
	\item[(co-4)] Any non zero object $E$ admits a Harder-Narasimhan filtration, that is a finite number of inclusion
	$$0=E_0 \hookrightarrow E_1 \hookrightarrow \cdots \hookrightarrow E_{n-1} \hookrightarrow E_n = E$$
	such that $F_j=\cone(E_{j-1} \hookrightarrow E_j)$ are semistable object of  with phase
	$$\phi(F_1) < \cdots < \phi(F_{n-1}) < \phi(F_n)$$
\end{enumerate}

The space $\costab(\dera)$ consisting of all the co-stability condition on a triangulated category $T$ is a topological manifold. By following \cite{COSTAB} and mimicking the proof of Theorem \ref{thm:Classification} one got the following proposition.

\begin{prop} \label{prop:Costab}
The co-stability manifold of $\dera$ is empty.
\end{prop}
\begin{proof}
As for the stability manifold of $\perfa$, it is sufficient to prove that there are no bounded co-$t$-structure in $\dera$. One can find, for example in \cite{MSS}, that such co-$t$-structures are in bijection with silting subcategories of $\dera$.
Recall that a subcategory $\mathcal{M}$ of a triangulated category $\T$ is called silting if $\hom_\T(\mathcal{M},\mathcal{M}[h])=0$ for every $h$ greater than zero and $\T$ is the smallest full triangulated category containing $\mathcal{M}$ and closed under direct summand.
Since all the non perfect indecomposable objects of $\dera$ are shifts of $X_\infty$ and $\hom(X_\infty,X_\infty[1]) \cong k$. It follows that all the possible stilting subcategories of $\dera$  are contained in $\perfa$. But then, since $\perfa$ is Karoubian closed, they can generate at most $\perfa$. Hence there are no silting subcategories of $\dera$.
\end{proof}


\medskip

{\small\noindent{\bf Acknowledgements.}
This paper was written while the authors were working on their Ph.D thesis. F.A. wants to thank the Dipartimento di Matematica ``F. Enriques'' of the Universit{\`a} degli Studi di Milano and the national research project "Geometria delle Variet\'a Proiettive" (PRIN 2010-11) for financial support. R.M. wants to thank the Dipartimento di Matematica  ``F. Casorati'' of the Universit{\`a} degli Studi di Pavia, the FAR 2010 (PV) \emph{"Variet\`a algebriche, calcolo algebrico, grafi orientati e topologici"} and INdAM (GNSAGA) for financial support. It is a pleasure to thank professors Alberto Canonaco and Paolo Stellari for their useful suggestions and helpful discussions. We want also to thank an anonymous referee for the careful reading of the manuscript and helpful comments and suggestions. 
}


\end{document}